\documentclass[a4paper,11pt]{article}
\usepackage[T1]{fontenc}
\usepackage{amsfonts}
\usepackage{textcomp}
\usepackage{amsmath, amsthm, amssymb, mathrsfs}
\usepackage{graphicx, url}
\usepackage{xcolor}
\usepackage[utf8]{inputenc}
\usepackage[margin=1.25in]{geometry}
\usepackage[font=scriptsize]{caption}

\title{Representation and Simulation of Multivariate Dickman Distributions and Vervaat Perpetuities}
\author{Michael Grabchak\footnote{Email address: {mgrabcha@uncc.edu}}\ \ and\  Xingnan Zhang\\
	{\it University of North Carolina Charlotte}}

\begin{document}

\newtheorem{prop}{Proposition}
\newtheorem{thrm}{Theorem}
\newtheorem{defn}{Definition}
\newtheorem{cor}{Corollary}
\newtheorem{lemma}{Lemma}
\newtheorem{remark}{Remark}
\newtheorem{exam}{Example}

\newcommand{\rd}{\mathrm d}
\newcommand{\GD}{\mathrm{GD}}
\newcommand{\MD}{\mathrm{MD}}
\newcommand{\MVP}{\mathrm{MVP}}
\newcommand{\rD}{\mathrm D}
\newcommand{\rE}{\mathrm E}
\newcommand{\rF}{\mathrm F}
\newcommand{\rP}{\mathrm P}
\newcommand{\rG}{\mathrm G}
\newcommand{\rM}{\mathrm M}
\newcommand{\ts}{\mathrm{TS}^p_{\alpha,d}}
\newcommand{\Exp}{\mathrm{Exp}}
\newcommand{\ID}{\mathrm{ID}}
\newcommand{\tr}{\mathrm{tr}}
\newcommand{\Var}{\mathrm{Var}}
\newcommand{\iid}{\stackrel{\mathrm{iid}}{\sim}}
\newcommand{\eqd}{\stackrel{d}{=}}
\newcommand{\cond}{\stackrel{d}{\rightarrow}}
\newcommand{\conv}{\stackrel{v}{\rightarrow}}
\newcommand{\conw}{\stackrel{w}{\rightarrow}}
\newcommand{\conp}{\stackrel{p}{\rightarrow}}
\newcommand{\simp}{\stackrel{p}{\sim}}
\maketitle

\setcounter{page}{1}
\pagenumbering{arabic}

\begin{abstract}
A multivariate extension of the Dickman distribution was recently introduced, but very few properties have been studied. We discuss several  properties with an emphasis on simulation. Further, we introduce and study a multivariate extension of the more general class of  Vervaat perpetuities and derive a number of properties and representations. Most of our results are presented in the even more general context of so-called $\alpha$-times self-decomposable distributions.\\

\noindent\textbf{Keywords:} Multivariate Dickman distribution; Multivariate Vervaat perpetuities; Self-decomposable distributions; Simulation
\end{abstract}

\section{Introduction}

The Dickman distribution arises in many applications, including in the study of random graphs, small jumps of L\'evy processes, and Hoare's quickselect algorithm. It is closely related to the Dickman function, which is important in the study of prime numbers. For details and many applications see the recent surveys \cite{Penrose:Wade:2004}, \cite{Molchanov:Panov:2020}, and \cite{Grabchak:Molchanov:Panov:2022}. The class of Vervaat perpetuities is closely related to the Dickman distribution and has applications in a variety of areas including economics, actuarial science, and astrophysics, see the references in \cite{Dassios:Qu:Lim:2019}. Recently, there has been much interest in the simulation of Dickman random variables and Vervaat perpetuities. This was studied in a series of papers, including  \cite{Devroye:2001}, \cite{Devroye:Fawzi:2010}, \cite{Fill:Huber:2010}, \cite{Chi:2012}, \cite{Cloud:Huber:2017}, and \cite{Dassios:Qu:Lim:2019}. A multivariate extension of the Dickman distribution was recently introduced in \cite{Bhattacharjee:Molchanov:2020}, where it arose in the context of a limit theorem related to certain point processes. However, very few properties of the multivariate Dickman distribution have been studied and, from what we have seen, a multivariate extension of Vervaat perpetuities has not been introduced. 

In this paper, we introduce a multivariate extension of the class of Vervaat perpetuities, which includes the multivariate Dickman distribution as a special case. We show that these distributions are infinitely divisible and, more specifically, that they coincide with the class of self-decomposable distributions with finite background driving L\'evy measures, no Gaussian part, and no drift. We have not seen this relationship mentioned previously even in the univariate case, although connections between the class of self-decomposable distribution and more general classes of perpetuities are known, see \cite{Jurek:1999}. In the interest of generality, our theoretical results are derived in the more general context of $\alpha$-times self-decomposable distributions. For these we derive three representations: as stochastic integrals, as shot noise, and as limits of certain triangular arrays. The latter two can be used for simulation. Further, for the multivariate Dickman distribution, we propose a third simulation method, which is based on a discretization of the spectral measure and is similar to the approach used in \cite{Xia:Grabchak:2022} for simulating multivariate tempered stable distributions.

The need for simulation methods is motivated by the fact that the Dickman distribution can be used to model the small jumps of large classes of L\'evy processes. In the univariate case this was shown in \cite{Covo:2009}; we will extend it to the multivariate case in a forthcoming work. The simulation of large jumps tends to be easier as they follow a compound Poisson distribution. Combining the two allows for the simulation of a wide variety of L\'evy processes, which can then be used for many applications. Our motivation comes from finance, where Monte Carlo methods based on L\'evy processes are often used for option pricing, see, e.g., \cite{Cont:Tankov:2004}. In particular, the ability to simulate multivariate L\'evy processes allows for the pricing of multi-asset, e.g., rainbow or basket options.

The rest of this paper is organized as follows. In Section \ref{sec: alpha times SD} we recall basic facts about $\alpha$-times self-decomposable distributions and give our main theoretical results. In Section \ref{sec: MD}, we formally introduce the multivariate Dickman distribution and multivariate Vervaat perpetuities. In Section \ref{sec: sim methods} we discuss three approximate simulation methods for the multivariate Dickman distribution with an emphasis on the bivariate case. A small-scale simulation study is conducted in Section \ref{sec: sim study}. Proof are postponed to Section \ref{sec: proofs}.

Before proceeding, we introduce some notation. We write $\mathbb R^d$ to denote the space of $d$-dimensional column vectors of real numbers, $\mathbb S^{d-1}=\{x\in\mathbb R^d:|x|=1\}$ to denote the unit sphere in $\mathbb R^d$, and $\mathfrak B(\mathbb R^d)$ and $\mathfrak B(\mathbb S^{d-1})$ to denote the classes of Borel sets on $\mathbb R^d$ and $\mathbb S^{d-1}$, respectively. For a distribution $\mu$ on $\mathbb R^d$, we write $\hat\mu$ to denote its characteristic function,  $X\sim \mu$ to denote that $X$ is a random variable with distribution $\mu$, and $X_1,X_2,\dots\iid \mu$ to denote that $X_1,X_2,\dots$ are independent and identically distributed (iid) random variables with distribution $\mu$. We write $U(0,1)$ to denote the uniform distribution on $(0,1)$, $\Exp(\lambda)$ to denote the exponential distribution with rate $\lambda$, $\delta_a$ to denote a point-mass at $a$, and $1_A$ to denote the indicator function on event $A$. We write $\Gamma$ to denote the gamma function, $\lfloor\cdot\rfloor$ to denote the floor function, and $\vee$ and $\wedge$ to denote the maximum and minimum, respectively. We write $\eqd$, $\cond$, and $\conw$ to denote equality in distribution, convergence in distribution, and weak convergence, respectively. For two sequences of real numbers $\{a_n\}$ and $\{b_n\}$, we write $a_n\sim b_n$ to denote $a_n/b_n\to1$ as $n\to \infty$. 

\section{$\alpha$-times self-decomposable distributions and main results}\label{sec: alpha times SD}

We begin by recalling that the characteristic function of an infinitely divisible distribution $\mu$ on $\mathbb R^d$ can be written in the form $\hat\mu(z) = \exp\{C_\mu(z)\}$, where
\begin{eqnarray*}\label{eq: char func inf div}
	C_{\mu}(z) =  -\langle z,Az\rangle + i \left\langle b, z \right\rangle +
	\int_{\mathbb{R}^d}\left(e^{i\left\langle z, x\right\rangle } - 1 - \left\langle z, x\right\rangle 1_{[|x|\le1]}\right) M(\rd x), 
	\ \ \ z \in \mathbb{R}^d,
\end{eqnarray*}
$A$ is a $d\times d$-dimensional covariance matrix called the Gaussian part, $b\in\mathbb R^d$ is the shift, and $M$  is the L\'evy measure, which is a Borel measure on $\mathbb R^d$ satisfying
\begin{equation*}\label{eq: levy measure equation gen}
M(\{0\}) = 0 \text{  and  } \int_{\mathbb{R}^d}(|x|^2 \wedge 1) M(\rd x) < \infty.
\end{equation*}
The parameters $A$, $M$, and $b$ uniquely determine this distribution and we write $\mu=\ID(A,M,b)$. We call $C_\mu$ the cumulant generating function (cgf) of $\mu$. Associated with every infinitely divisible distribution $\mu$ is a L\'evy process $\{X_t:t\ge0\}$, where $X_1\sim\mu$. This process has finite variation if and only if $A=0$ and $M$ satisfies the additional assumption 
\begin{equation}\label{eq: levy measure equation finite var}
\int_{\mathbb{R}^d}(|x| \wedge 1) M(\rd x) < \infty.
\end{equation}
Through a slight abuse of terminology, we also say that the associated distribution $\mu$ has finite variation. In this case, the cgf 
can be written in the form
\begin{eqnarray}\label{eq: char func inf div finite variation}
	C_\mu(z)=  i \left\langle \gamma, z \right\rangle +
	\int_{\mathbb{R}^d}\left(e^{i\left\langle z, x\right\rangle } - 1\right) M(\rd x),
	 \ \ \ z \in \mathbb{R}^d,
\end{eqnarray}
where $\gamma = b-\int_{|x|\le1}x M(\rd x) \in \mathbb{R}^d$ is the drift and we write $\mu = \ID_0(M, \gamma)$.

A distribution $\mu$ is said to be self-decomposable if for any $c\in(0,1)$ there exists a probability measure $\rho_c$ with
\begin{eqnarray}\label{eq: self dec} 
\hat\mu(z) = \hat\mu(cz)\hat\rho_c(z), \ \ z\in\mathbb R^d.
\end{eqnarray}
Equivalently, if $X\sim\mu$ and $Y_c\sim\rho_c$, then $X\eqd cX+Y_c$, where $X$ and $Y_c$ are independent on the right side. We denote the class of self-decomposable distributions by $L_1$. These distributions are important in the study of stationary $\mathrm{AR}(1)$ processes and the limits of sums of independent random variables. Next, for $\alpha\in\{2,3,\dots\}$ we define the classes $L_\alpha$ recursively, as follows. A distribution $\mu\in L_\alpha$ if for every $c\in(0,1)$ there exists a $\rho_c\in L_{\alpha-1}$ such that \eqref{eq: self dec} holds. There is substantial literature on the study of these classes, see, e.g., the monograph \cite{Rocha-Arteaga:Sato:2019} and the references therein. The distributions in $L_\alpha$ are sometimes called $\alpha$-times self-decomposable. These should not be confused with the so-called $\alpha$-self-decomposable distributions, studied in, e.g., \cite{Maejima:Ueda:2010} and the references therein.

It is well-known that every $\mu\in L_\alpha$ is infinitely divisible. In fact, $\mu\in L_\alpha$ if and only if $\mu=\ID(A,M,b)$ with $M=M_\alpha$, where
\begin{eqnarray}\label{eq: M alpha}
M_\alpha(B) = \int_{\mathbb R^d}\int_0^1 1_{B}(yr) (-\log r)^{\alpha-1}r^{-1}\rd r \nu(\rd y), \ \ \ B\in\mathfrak B(\mathbb R^d),
\end{eqnarray}
for some Borel measure $\nu$ satisfying 
\begin{eqnarray}\label{eq: finite for MVP alpha}
\nu(\{0\})=0\ \mbox{ and }\ \int_{|x|\le2}|x|^2\nu(\rd x)+\int_{|x|>2}\left(\log|x|\right)^\alpha\nu(\rd x)<\infty.
\end{eqnarray}
Clearly, $M_\alpha$ can be a L\'evy measure even if $\alpha$ is not an integer. In fact, $M_\alpha$ is a L\'evy measure for any $\alpha>0$ so long as \eqref{eq: finite for MVP alpha} holds. The study of this extension to non-integer $\alpha$ was initiated in \cite{Thu:1982}, see also \cite{Sato:2010} and the references therein. This leads to the following.

\begin{defn}
Fix $\alpha\in(0,\infty)$. We write $L_\alpha$ to denote the class of all infinitely divisible distributions $\ID(A,M,b)$ where $M=M_\alpha$ is of the form \eqref{eq: M alpha} for some Borel measure $\nu$ on $\mathbb R^d$ satisfying \eqref{eq: finite for MVP alpha}. We refer to $L_\alpha$ as the class of $\alpha$-times self-decomposable distributions. The measure $\nu$ is called the background driving L\'evy measure (BDLM).
\end{defn}

The name BDLM is motivated by the role that this measure plays in the context of OU-processes, see \cite{Rocha-Arteaga:Sato:2019}.  We now give a result that relates certain moment properties of $M_\alpha$ to those of the BDLM $\nu$. 

\begin{lemma}\label{lemma: moments}
Fix $\alpha\in(0,\infty)$ and let $M_\alpha$ be as in \eqref{eq: M alpha}, where $\nu$ is a Borel measure on $\mathbb R^d$ satisfying \eqref{eq: finite for MVP alpha}. In this case we can equivalently write
\begin{eqnarray}\label{eq: M alpha 2}
M_\alpha(B) &=&  \int_{\mathbb R^d}\int_0^\infty 1_{B}(ye^{-r}) r^{\alpha-1}\rd r \nu(\rd y)\nonumber\\
&=&  \theta^{-1}\int_{\mathbb R^d}\int_0^\infty 1_{B}(ye^{-(\alpha r/\theta)^{1/\alpha}})\rd r \nu(\rd y), \ \ \ B\in\mathfrak B(\mathbb R^d)
\end{eqnarray}
for any $\theta>0$. Further, $M_\alpha(\mathbb R^d)=\infty$ for any $\nu\ne0$ and for any $p>0$, we have
$$
\int_{|x|\le1}|x|^p M_\alpha(\rd x)<\infty
\mbox{ if and only if }
\int_{|x|\le1}|x|^p \nu(\rd x)<\infty
$$
and 
$$
\int_{|x|>1}|x|^p M_\alpha(\rd x)<\infty
\mbox{ if and only if }
\int_{|x|>1}|x|^p \nu(\rd x)<\infty.
$$
\end{lemma}

Combining the lemma with Proposition 25.4 in \cite{Sato:1999} shows that for $X\sim\mu\in L_\alpha$ with BDLM $\nu$ and any $p>0$ we have
$$
\rE|X|^p<\infty \mbox{ if and only if }
\int_{|x|>1}|x|^p \nu(\rd x)<\infty.
$$
In the context of multivariate Vervaat perpetuities, we are only interested in distributions with no Gaussian part and a finite BDLM.  In this case, Lemma \ref{lemma: moments} imples that $M_\alpha$ will satisfy \eqref{eq: levy measure equation finite var}, which leads to the following.

\begin{defn}
Fix $\alpha\in(0,\infty)$. We write $L^*_\alpha$ to denote the class of all infinitely divisible distribution $\mu=\ID_0(M,\gamma)$ where $M=M_\alpha$ is of the form \eqref{eq: M alpha} for some finite Borel measure $\nu$ on $\mathbb R^d$ satisfying \eqref{eq: finite for MVP alpha}. In this case, we write $\mu=L^*_\alpha(\nu,\gamma)$. 
\end{defn}

From \eqref{eq: char func inf div finite variation} and \eqref{eq: M alpha} it follows that the cgf of $L^*_\alpha(\nu,\gamma)$ is given by
\begin{eqnarray*}\label{eq: char func L*}
i\langle \gamma,z\rangle+\int_{\mathbb{R}^{d}}\int_0^1 \left(e^{ir\left\langle z, y\right\rangle } - 1\right)(-\log r)^{\alpha-1}r^{-1}\rd r \nu(\rd y), \ \ \ z \in \mathbb{R}^d.
\end{eqnarray*}
Taking partial derivatives shows that, when they exist, the mean vector and covariance matrix of $X\sim L^*_\alpha(\nu,\gamma)$ are given by
\begin{eqnarray}\label{eq: mean and var L*}
\rE[X] =\gamma+ \Gamma(\alpha) \int_{\mathbb R^d} y \nu(\rd y) \mbox{  and  }
\mathrm{cov}(X) = \frac{\Gamma(\alpha)}{2^\alpha}\int_{\mathbb R^{d}} yy^T \nu(\rd y).
\end{eqnarray}

Next, we give three representations of the distributions in $L^*_\alpha$. The first as a stochastic integral, the second as shot noise, and the third as the limit of a triangular array. The first result is essentially contained in \cite{Sato:2010}, but, for completeness and due to a difference in presentation, a self-contained proof is given in Section \ref{sec: proofs}. 

\begin{thrm}\label{thrm: integ rep}
Let $X\sim L^*_\alpha(\nu,\gamma)$ and fix $\theta>0$. If $\{Y_t:t\ge0\}$ is a L\'evy process with $Y_1\sim\ID_0(\nu',\gamma')$, where $\nu'=\nu/\theta$ and  $\gamma'=\gamma/(\theta\Gamma(\alpha))$, then 
$$
X \eqd \int_0^\infty e^{-\left(\frac{\alpha s}{\theta}\right)^{1/\alpha}} \rd Y_s,
$$
where the stochastic integral is absolutely definable in the sense of \cite{Sato:2006}.
\end{thrm}

Our second representation is as an infinite series. This is sometimes called a shot noise representation. It is not just for one random variable, but for the corresponding L\'evy process.

\begin{thrm}\label{thrm: series rep gen Vervaat}
Fix $\alpha\in(0,\infty)$, let $\nu$ be a finite nonzero measure satisfying \eqref{eq: finite for MVP alpha}, and set $\theta=\nu(\mathbb R^d)$ and $\nu_1=\nu/\theta$. Let $E_1,E_2,\dots\iid \mathrm{Exp}(1)$, $V_1,V_2,\dots\iid U(0,1)$, and $Y_1,Y_2,\dots\iid \nu_1$ be independent sequences of random variables and let $\Gamma_i = E_1+E_2+\cdots+E_i$. Fix $\gamma\in\mathbb R^d$, $T>0$ and set
\begin{eqnarray}\label{eq: shot noise for Levy meas}
X_t = t\gamma+ \sum_{i=1}^{\infty} e^{-\left(\frac{\Gamma_i}{T\theta}\right)^{1/\alpha}} Y_i 1_{\left[0,\frac{t}{T}\right]}(V_i), \ \ \ t\in[0,T].
\end{eqnarray}
Then the series converges almost surely and uniformly on $t\in[0,T]$ and $\{X_t:0\le t\le T\}$ is a L\'evy process with $X_t\sim L^*_\alpha(t\nu,t\gamma)$. When $\alpha=1$ we can write \eqref{eq: shot noise for Levy meas} as
\begin{eqnarray}\label{eq: shot noise for Levy meas alpha=1}
X_t = t\gamma+ \sum_{i=1}^{\infty} \left(U_1U_2\cdots U_i\right)^{1/(T\theta)} Y_i 1_{\left[0,\frac{t}{T}\right]}(V_i), \ \ \ t\in[0,T],
\end{eqnarray}
where $U_1,U_2,\dots\iid U(0,1)$ are independent of the sequences of $Y_i$'s and $V_i$'s.
\end{thrm}

In light of \eqref{eq: finite for MVP alpha}, Theorem \ref{thrm: series rep gen Vervaat} implicitly assumes that $\nu_1(\{0\})=P(Y_i=0)=0$. However, the representation in \eqref{eq: shot noise for Levy meas} holds even if this is not the case. To see this, assume that $X_t'$ is given as in \eqref{eq: shot noise for Levy meas}, but with some $\gamma'$ in place of $\gamma$ and $Y_i'$ in place of $Y_i$, where $Y_i'\sim \nu'_1$ and $\nu_1'(\{0\})=u$ for some $u\in(0,1)$. Thus,
\begin{eqnarray*}\label{eq: shot noise for Levy meas 2}
X_t' = t\gamma'+ \sum_{i=1}^{\infty} e^{-\left(\frac{\Gamma_i}{T\theta}\right)^{1/\alpha}} Y_i' 1_{\left[0,\frac{t}{T}\right]}(V_i).
\end{eqnarray*}
By independence $P([Y'_i=0]\cup[V_i>t/T]) = 1-(1-u)t/T$. Thus,
$$
Y_i'1_{\left[0,\frac{t}{T}\right]}(V_i)\eqd Y_i1_{\left[0,\frac{(1-u)t}{T}\right]}(V_i),
$$
where $Y_i\sim \nu_1$ is independent of $V_i$ and $\nu_1 (\rd y)= 1_{[|y|>0]}\nu_1'(\rd y)/(1-u)$ is a probability measure. It follows that $X_t'\eqd X_{t(1-u)}$, where $X_{t(1-u)}$ is given as in \eqref{eq: shot noise for Levy meas} if we take $\gamma=\gamma'/(1-u)$ and $\nu(\rd y) =\theta 1_{[|y|>0]}\nu_1'(\rd y)/(1-u)$ in Theorem \ref{thrm: series rep gen Vervaat}. Hence, $X_t' \sim L^*_\alpha(t\nu',t\gamma')$, where $\nu'(\rd y) = \theta1_{[|y|>0]}\nu_1'(\rd y)$.

We now turn to the third representation, which is as the limiting distribution of a triangular array. Let $X_1,X_2,\dots$ be iid random variables with support contained in $[0,1]$ and 
$$
P(X_1>x) = (1-x)^\alpha \ell(1-x), \ x\in[0,1],
$$
where $\alpha>0$ and $\ell$ is a slowly varying at $0$ function. This means that for every $t>0$
$$
\lim_{x\to0^+} \frac{\ell(xt)}{\ell(x)}=1.
$$

\begin{thrm}\label{thrm: conv of powers}
Let $\nu_0$ be a distribution on $\mathbb R^d$ and let $T_1,T_2,\dots\iid\nu_0$ be independent of the sequence of $X_i$'s. Assume that $\ell$ is bounded away from $0$ and $\infty$ on every compact subset of $(0,1]$ and that there exists a $\gamma\in(0,1)$ with $\rE|T_1|^\gamma<\infty$. Let $N_n$ be a sequence of integers with $N_nn^{-\alpha}\ell(1/n)\to c$ for some $c\in(0,\infty)$ and set
$$
A_n =\sum_{i=1}^{N_n} T_i X_i^{n}.
$$
Then \(A_n \cond A_\infty\), where \(A_\infty\sim L^*_\alpha(\nu,0)\) with $\nu(\rd x) = c1_{[|x|>0]}\nu_0(\rd x)$.
\end{thrm}

In the univariate case, versions of this result were studied in several papers. In \cite{Schlather:2001} it was studied in the context of limits of $\ell^p$ norms of random vectors as both $p$ and the dimension of the vector approach infinity. In \cite{Grabchak:Molchanov:2019} it was studied in the context of the so-called random energy model (REM), which is important in statistical physics. In both papers it is assumed that $P(T_1=1)=1$. While other distributions were considered in \cite{Molchanov:Panov:2020} and \cite{Grabchak:Molchanov:Panov:2022}, in all of these univariate results, it is assumed that $\nu_0$ either has a bounded support or exponential moments. Here we make a much weaker assumption on the tails. Thus, this result is new even in the one-dimensional case.

\section{Multivariate Dickman Distribution and Vervaat Perpetuities}\label{sec: MD}

In the univariate case, a positive random variable $X$ is said to have a generalized Dickman (GD) distribution if
\begin{equation}\label{eq: dickman equation}
X\eqd U^{1/\theta}(X+1),
\end{equation}
where $\theta>0$ and $U\sim U(0,1)$ is independent of $X$ on the right side. We denote this distribution by $\GD(\theta)$.  When $\theta=1$, it is just called the Dickman distribution. A multivariate extension of this distribution was recently introduced in \cite{Bhattacharjee:Molchanov:2020}.  It is defined as follows.
 
 \begin{defn}
 Let $\sigma_1$ be a probability measure on $\mathbb S^{d-1}$ and let $W\sim \sigma_1$. A random variable $X$ on $\mathbb R^d$ is said to have a multivariate Dickman (MD) distribution if for some $\theta>0$
\begin{eqnarray}\label{eq: relation for MD}
X\eqd U^{1/\theta}( X+ W),
\end{eqnarray}
where $U\sim U(0,1)$ and $X, W, U$ are independent on the right side. We denote this distribution $\MD(\sigma)$, where $\sigma=\theta\sigma_1$. We call $\sigma$ the spectral measure.
\end{defn}

There is no loss of information when working with $\sigma$ instead of $\theta$ and $\sigma_1$ since $\theta=\sigma(\mathbb S^{d-1})$ and $\sigma_1 = \sigma/\theta$. When the dimension $d=1$, $\sigma(\{-1\})=0$ and $\sigma(\{1\})=\theta>0$, we have $\MD(\sigma)=\GD(\theta)$. More generally, when $d=1$, $\sigma(\{1\})=\theta_1>0$, and $\sigma(\{-1\})=\theta_2>0$, it is easily checked that $\MD(\sigma)$ is the distribution $X_1-X_2$, where $X_1\sim\GD(\theta_1)$ and $X_2\sim\GD(\theta_2)$. 

We now turn to Vervaat perpetuities, which are named after the author of \cite{Vervaat:1979}. In the univariate case a distribution $\mu$ on $[0,\infty)$ is said to be a Vervaat perpetuity if there exists a $\theta>0$ and a distribution $\nu_1$ on $[0,\infty)$ such that, if $X\sim\mu$, $Z\sim\nu_1$, and $U\sim U(0,1)$, then 
 \begin{eqnarray}\label{eq: MVP defn}
X\eqd U^{1/\theta}\left(X+Z\right),
\end{eqnarray}
where $X,Z,U$ are independent on the right. It can be shown that a solution exists if and only if 
\begin{eqnarray}\label{eq: finite for MVP}
\int_{|x|>2}\log|x|\nu_1(\rd x)<\infty.
\end{eqnarray}
We now extend this idea to the multivariate case.

\begin{defn}
Fix $\theta>0$ and let $\nu_1$ be a probability measure on $\mathbb R^d$ satisfying \eqref{eq: finite for MVP}. The distribution of a random variable $X$ on $\mathbb R^d$ is said to be a multivariate Vervaat perpetuity (MVP) if \eqref{eq: MVP defn} holds, where $Z\sim\nu_1$, $U\sim U(0,1)$, and $U,X,Z$ are independent on the right side. 
\end{defn}

From \eqref{eq: MVP defn} it follows that the distribution of $X$ is MVP if and only if
\begin{eqnarray}\label{eq: main sum for MVP}
X\eqd Z_1 U_1^{1/\theta}+ Z_2 \left(U_1U_2\right)^{1/\theta}+Z_3\left(U_1U_2U_3\right)^{1/\theta} +\cdots,
\end{eqnarray}
where $Z_1,Z_2,\dots\iid\nu_1$ and $U_1,U_2,\dots\iid U(0,1)$ are independent sequences. By comparing \eqref{eq: main sum for MVP} with \eqref{eq: shot noise for Levy meas alpha=1} and taking into account the discussion just below Theorem \ref{thrm: series rep gen Vervaat}, it follows that $X\sim L^*_1(\nu,0)$, where 
$\nu(\rd x) = \theta 1_{[|x|>0]}\nu_1(\rd x)$. From here we get the following.

\begin{thrm}
$\mu$ is $\MVP$ if and only if $\mu= L_1^*(\nu,0)$ for some finite measure $\nu$.
\end{thrm}

To the best of our knowledge, this result was previously unknown, even in the univariate case. Next, comparing \eqref{eq: relation for MD} and \eqref{eq: MVP defn} shows that MD distributions are special cases of MVP and hence of $L_1^*$. More specifically, we immediately get the following.

\begin{thrm}\label{thrm: char MD}
$\mu$ is $\MD$ if and only if $\mu=L_1^*(\nu,0)$ where $\nu(\mathbb R^d\setminus\mathbb S^{d-1}) = 0$. In this case $\mu=\MD(\sigma)$, where $\sigma=\nu$.
\end{thrm}

It follows that all of the results in Section \ref{sec: alpha times SD} specialize to MVP and MD distributions. In particular Theorems \ref{thrm: series rep gen Vervaat} and \ref{thrm: conv of powers} can be used for approximate simulation.

\section{Simulation from Multivariate Dickman Distributions}\label{sec: sim methods}

In this section we focus on the simulation of MD random variables. First, combining Theorem \ref{thrm: char MD} with \eqref{eq: mean and var L*} shows that the mean vector and covariance matrix of $X\sim\MD(\sigma)$ are given by
\begin{eqnarray}\label{eq: mean and var MD}
\rE[X] = \int_{\mathbb S^{d-1}} s \sigma(\rd s) \mbox{  and  }
\mathrm{cov}(X) = \frac{1}{2}\int_{\mathbb S^{d-1}} ss^T \sigma(\rd s).
\end{eqnarray}

As mentioned, Theorems \ref{thrm: series rep gen Vervaat} and \ref{thrm: conv of powers} can be used to develop approximate simulation methods.   We now derive another method, which is exact in the important case when the spectral measure has finite support.  This means that there is a positive integer $k$ such that the spectral measure is given by
\begin{eqnarray}\label{eq: sigma k}
\sigma_k = \sum_{i=1}^k a_i \delta_{s_i},
\end{eqnarray}
where $s_1,s_2,\dots,s_k\in\mathbb S^{d-1}$ and $a_1,a_2,\dots,a_k\in(0,\infty)$. In this case, it is readily checked that 
$$
\sum_{i=1}^k s_i Y_i \sim \MD(\sigma_k),
$$
where $Y_1,Y_2,\dots,Y_k$ independent random variables with $Y_i\sim \GD(a_i)$ for $i=1,2,\dots,k$. Thus, to simulate from $\MD(\sigma_k)$ we just need a way to simulate from $\GD$. Exact simulation methods for $\GD$ are available, see, e.g., \cite{Devroye:2001}, \cite{Devroye:Fawzi:2010}, \cite{Fill:Huber:2010}, \cite{Chi:2012}, \cite{Cloud:Huber:2017}, and \cite{Dassios:Qu:Lim:2019}. In this paper, we use the method of \cite{Dassios:Qu:Lim:2019}, which is implemented in the SubTS \cite{Grabchak:Cao:2023} package for the statistical software R. 

This approach can be modified to work even when the support of the spectral measure is infinite. The idea is that for any spectral measure $\sigma$ there exists a sequence of spectral measures $\{\sigma_k\}$ on $\mathbb S^{d-1}$, each having finite support, such that $\MD(\sigma_k) \conw \MD(\sigma)$. This follows immediately from Theorem 7.1 in \cite{Xia:Grabchak:2022}. Thus, we can approximately simulate from $\MD(\sigma)$ by first discretizing $\sigma$ and approximating it by some $\sigma_k$ with finite support. We can then use the above approach to simulate from $\MD(\sigma_k)$, which is an approximate simulation from $\MD(\sigma)$. We call this the discretization and simulation (DS) method. A version of this approach was used in \cite{Xia:Grabchak:2022} to simulate from certain multivariate tempered stable distributions.

For the remainder of this section, we specialize our results and give additional details in the important case of bivariate distributions. Here, every point $s\in\mathbb S^1$ can be written as $s=(\cos\phi,\sin\phi)$ for some angle $\phi\in[0,2\pi)$. Thus, corresponding to the measure $\sigma$ on $\mathbb S^1$, there is a measure $\sigma'$ on $[0,2\pi)$ satisfying
$$
\sigma(B) = \int_{[0,2\pi)} 1_B((\cos\phi,\sin\phi)) \sigma'(\rd \phi), \ \ B\in\mathfrak B(\mathbb S^1).
$$
From \eqref{eq: mean and var MD}, it follows that if $X=(X_1,X_2)\sim\MD(\sigma)$, then in terms of $\sigma'$, we can write
$$
\rE[X_1] = \int_{[0,2\pi)} \cos\phi\ \sigma'(\rd\phi), \quad \rE[X_1] = \int_{[0,2\pi)} \sin\phi\ \sigma'(\rd\phi),
$$
$$
\mathrm{Var}(X_1) = \frac{1}{2} \int_{[0,2\pi)} \cos^2\phi\ \sigma'(\rd\phi), \quad \mathrm{Var}(X_2) = \frac{1}{2}  \int_{[0,2\pi)} \sin^2\phi\ \sigma'(\rd\phi),  
$$
and 
$$
\mathrm{Cov}(X_1,X_2) = \frac{1}{2}  \int_{[0,2\pi)} \cos\phi \sin\phi\ \sigma'(\rd\phi).
$$

Next, we specialize our three simulation methods to this case. Toward this end, let $\theta=\sigma(\mathbb S^1) = \sigma'([0,2\pi))$ and let $\sigma_1 = \sigma/\theta$ and $\sigma_1'=\sigma'/\theta$ be probability measures. Throughout, we assume that we know how to simulate from $\sigma'_1$. All of our simulation methods depend on a tuning parameter $k$. The first method is based on the shot-noise representation given in Theorem \ref{thrm: series rep gen Vervaat} and is denoted SN. For this approximation we take the first $k$ terms in the series, which gives
$$
\sum_{i=1}^{k} \left(U_1U_2\cdots U_i\right)^{1/\theta} \zeta_i,
$$
where $U_1,U_2,\dots,U_k\iid U(0,1)$ and $\zeta_1,\zeta_2,\dots,\zeta_k$ are iid and are simulated by taking $\zeta_i=(\cos(\phi_i),\sin(\phi_i))$, where $\phi_i\sim\sigma_1'$. The second method is based on the triangular array approximation given in Theorem \ref{thrm: conv of powers} and is denoted TA. Here we evaluate $A_k$ for some large $k$. We take $\alpha=1$, $\ell(x)=1$ for all $x$, $N_k=\lfloor k\theta\rfloor$, 
and $c=\theta$, which gives
$$
\sum_{i=1}^{\lfloor k\theta\rfloor} U_i^k \zeta_i,
$$
where $U_1,U_2,\dots,U_k\iid U(0,1)$ and the $\zeta_i$ are as in the SN method. Finally, for the DS method, we assume 
either that $\sigma$ is of the form given in \eqref{eq: sigma k} or that we have an approximation $\sigma_k$ of $\sigma$ that is of this form. Either way, we have
$$
\sum_{i=1}^{k} Y_i s_i,
$$
where $Y_1,Y_2,\dots,Y_k$ independent random variables with $Y_i\sim \GD(a_i)$ for $i=1,2,\dots,k$ and $a_i,s_i$ are as in  \eqref{eq: sigma k}. Note that, in the first two methods, the directions are random, while in the third they are deterministic.

All that remains is to describe a systematic approach for discretizing $\sigma$. We start by selecting integer $k$, which is the number of terms in the support of the approximation, selecting $0=d_0<d_1<\cdots<d_k=2\pi$, and selecting $\phi_1,\phi_2,\dots,\phi_k$ with $d_{i-1}\le \phi_i<d_i$ for each $j=1,2,\dots,k$. Next we take $a_i= \sigma'([d_{i-1},d_i))$. We then approximate $\sigma$ and $\sigma'$ by
$$
\sigma_k = \sum_{i=1}^k a_i \delta_{s_i} \mbox{ and } \sigma_k' = \sum_{i=1}^k a_i \delta_{\phi_i}, 
$$
where $s_i=(\cos\phi_i,\sin\phi_i)$. For simplicity, in this paper, we take $d_i =2\pi i/k$ to be evenly spaced and $\phi_i = d_{i-1}$.

\section{Simulation Study}\label{sec: sim study}

In this section we perform a small-scale simulation study to compare the performance of the three methods discussed in Section \ref{sec: sim methods}. For simplicity, we focus on the bivariate case. In this context, we consider two models for the spectral measure: in the first the spectral measure is a beta distribution and in the second it has finite support.  In the latter case, the DS method is exact.

We begin with the first model. Here the spectral measure depends on two parameters $\alpha,\beta>0$ and is denoted by $\sigma^{\alpha,\beta}$. We assume that $\sigma^{\alpha,\beta}$ is a probability measure on $\mathbb S^1$ such that it is the distribution of the random vector $\zeta=(\cos(\phi),\sin(\phi))$ where $\phi$ has a beta distribution on $[0,2\pi)$, i.e., the distribution of $\phi$ has a density given by
$$
f(x) = \frac{(2\pi)^{1-\alpha - \beta}}{B(\alpha, \beta)} x^{\alpha -1}(2\pi - x)^{\beta - 1}, \ \ 0\le x<2\pi,
$$
where $B$ is the beta function. When $\alpha=\beta=1$ this reduces to the uniform distribution on $[0,2\pi)$. Since $\sigma^{\alpha,\beta}$ is a probability measure, we have $\theta=1$, and since the support of $\sigma^{\alpha,\beta}$ is infinite, all three simulation methods are approximate. They depend on a tuning parameter $k$ and, as $k$ increases, all methods get closer to simulating from $\MD(\sigma^{\alpha,\beta})$. Our goal is to understand which methods converge faster.

\begin{figure}
\center
\begin{tabular}{cc}
$\alpha=1,\beta=1$ & $\alpha=2,\beta=2$\vspace{-.35cm} \\
\includegraphics[width=0.45\textwidth]{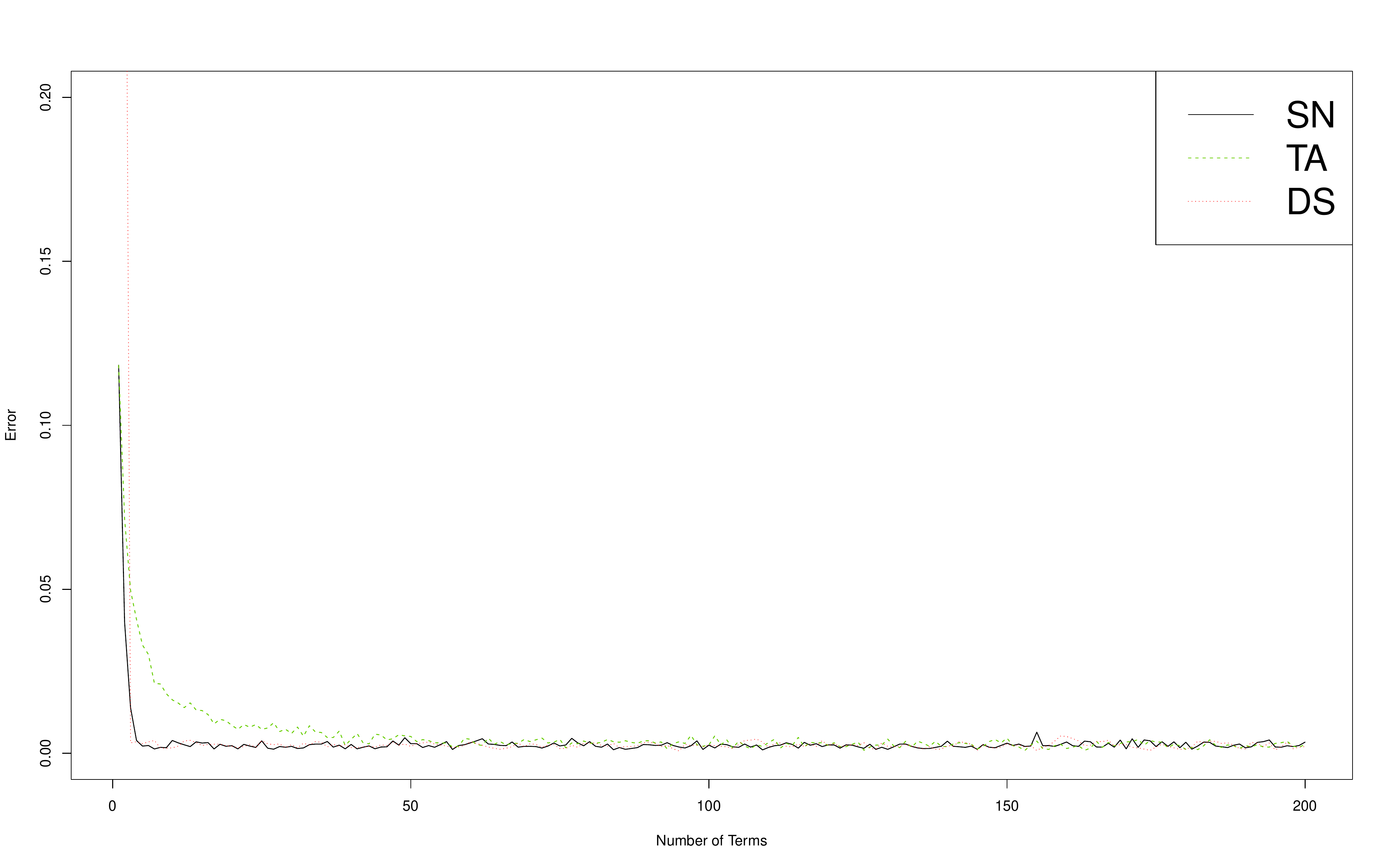} & \includegraphics[width=0.45\textwidth]{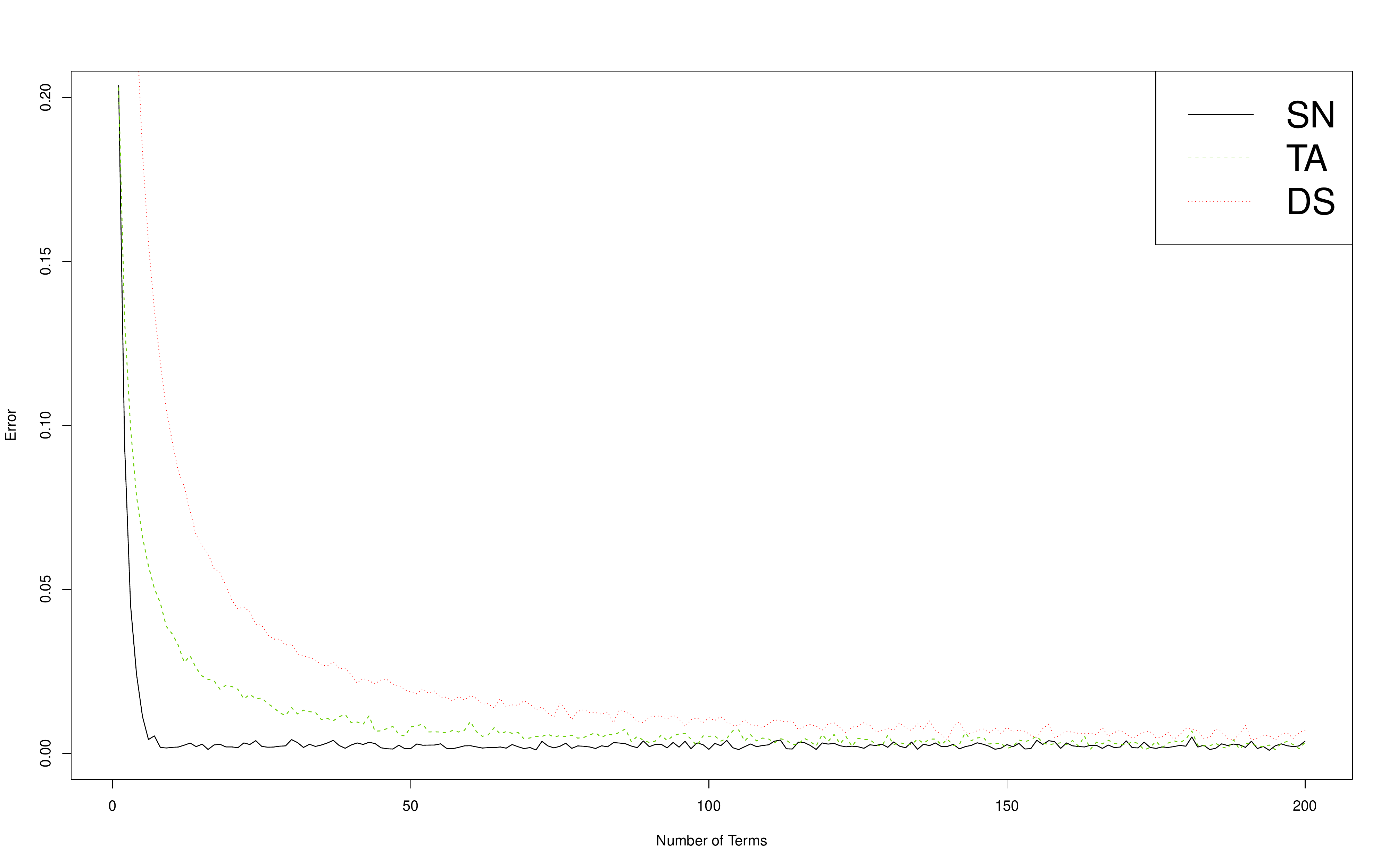} \\
$\alpha=2,\beta=5$ & $\alpha=5,\beta=1$ \vspace{-.35cm} \\
\includegraphics[width=0.45\textwidth]{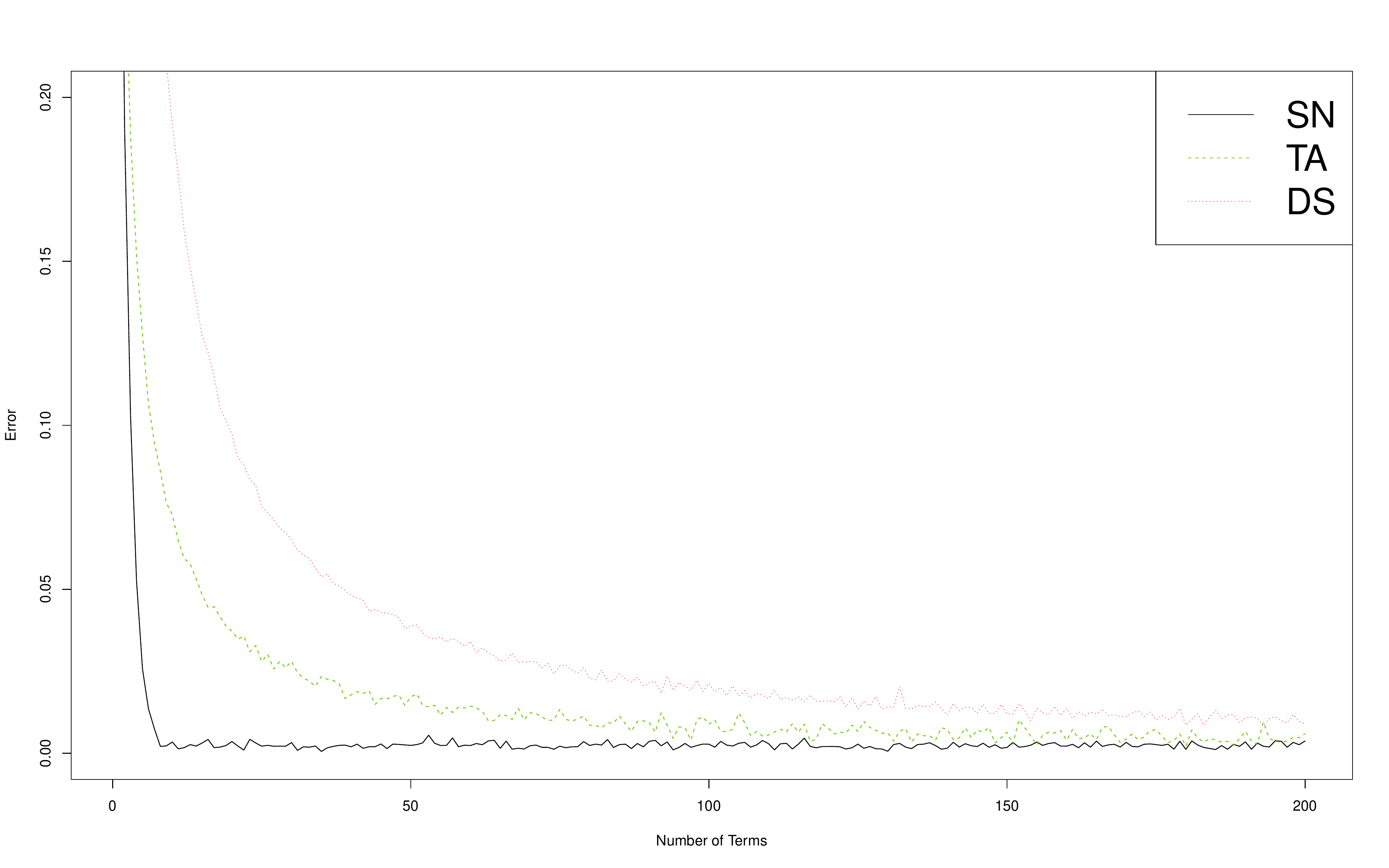} & \includegraphics[width=0.45\textwidth]{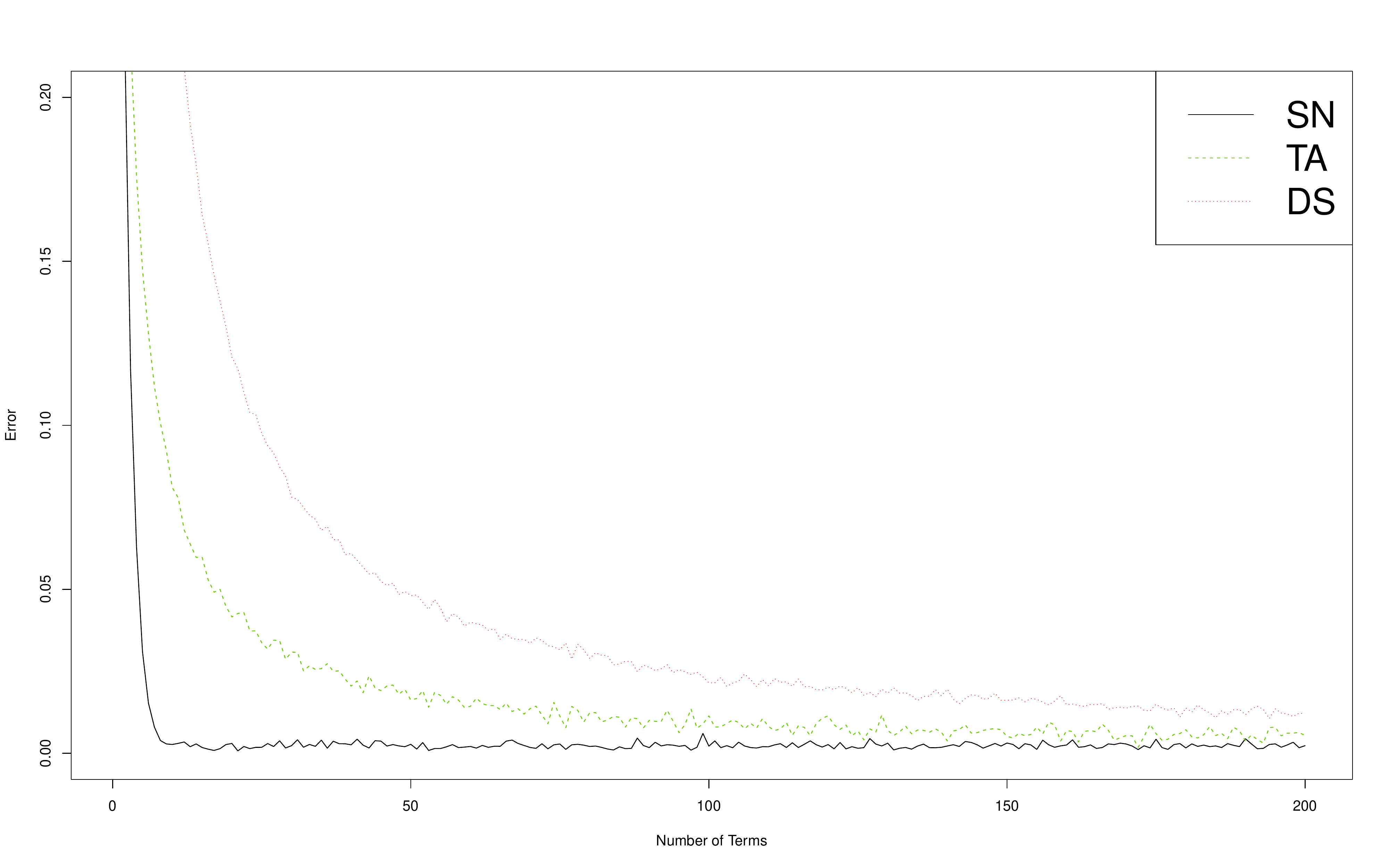}
\end{tabular}
\caption{Plots of errors in the beta model with all three methods and several choices of the parameters. The $x$-axis represents $k$, the number of terms in the sum, and the $y$-axis represents the errors.}\label{fig: error in sims}
\end{figure}

Our simulations are performed as follows. We begin by choosing values for $\alpha$ and $\beta$. We then select one of our three approximate simulation methods and a value for the tuning parameter $k$. Next, we use the method to approximately simulate $N$ observations from $\MD(\sigma^{\alpha,\beta})$. Using these, we estimate the means $m_1$, $m_2$ of both components, the variances $\sigma_1^2$, $\sigma^2_2$ of both components, and the covariance $\sigma_{12}$ between the components by using the empirical means $\bar x_1$, $\bar x_2$, the empirical variances $s_1^2$, $s_2^2$, and the empirical covariance $s_{12}$, respectively. We then quantify the error in the approximation by
\begin{equation}\label{eq:total error}
\mathrm E_k = \sqrt{(\bar x_1 - m_1)^2 + (\bar x_2 - m_2)^2 + (s_1^2 - \sigma_1^2)^2 + (s_2^2 - \sigma_2^2)^2 + (s_{12} - \sigma_{12})^2}.
\end{equation}
This was used to quantify errors in a similar context in \cite{Xia:Grabchak:2022}. The values of $m_1$, $m_2$, $\sigma_1^2$, $\sigma^2_2$, and $\sigma_{12}$ can be calculated by numerically integrating the formulas given in Section \ref{sec: sim methods}. When $\alpha$ and $\beta$ are integers, one can also evaluate the integrals explicitly using integration by parts. Note that only part of the error is due to the fact that the methods are approximate, the other part is due to Monte Carlo error as we are using only a finite number ($N$) of replications.

The results of our simulations are presented in Figure \ref{fig: error in sims}. Here we consider four combinations of the parameters $(\alpha, \beta)$: $(1, 1)$, $(2, 2)$, $(2, 5)$, and $(5,1)$. For each  method and each choice of the parameters, we let $k$ range from $1$ to $200$. In each case, we simulate $N=160000$ replications. Figure \ref{fig: error in sims} presents the value of $k$ plotted against the error $\mathrm E_k$. When $\alpha=\beta=1$, which corresponds to a uniform distribution, the SN and DS methods have similar performance and the error deceases very quickly. In comparison, the error for TA decreases much slower.  For the other cases, SN has the best performance, followed by TA, and then DS has the worst performance.

Next, we turn to our second model. Here $\sigma$ has a finite support, i.e, $\sigma = \sum_{i=1}^{r} a_i \delta_{s_i}$. For simplicity, we take $a_i=1/r$ for each $i=1,2,\dots,r$ and the $s_i$'s to be evenly spaced. We again quantify the error using $\mathrm E_k$ as given in \eqref{eq:total error} and use the formulas in Section \ref{sec: sim methods} to calculate the means, variances, and the covariance. This time they are all finite sums. The results of our simulations are given in Figures \ref{fig: error in sims discrete} and \ref{fig: error in sims discrete diff dir}. In Figure \ref{fig: error in sims discrete} we consider the case where $r=50$. Since the DS method is exact in this case, we just evaluate it once and plot the resulting error as a baseline.  Note that, due to Monte Carlo error, this error is not zero. For the other methods, we see that the error in SN decays quickly as $k$ increases, whereas for TA it decays slower. To get an idea of how the performance of the methods depends on the number of terms $r$, in Figure \ref{fig: error in sims discrete diff dir} we consider the case where $r=2,20,100$. We can see that both SN and TA work better when $r$ is larger. We do not include DS in these simulations as it is exact in this case.

\begin{figure}
\center
\includegraphics[width=0.45\textwidth]{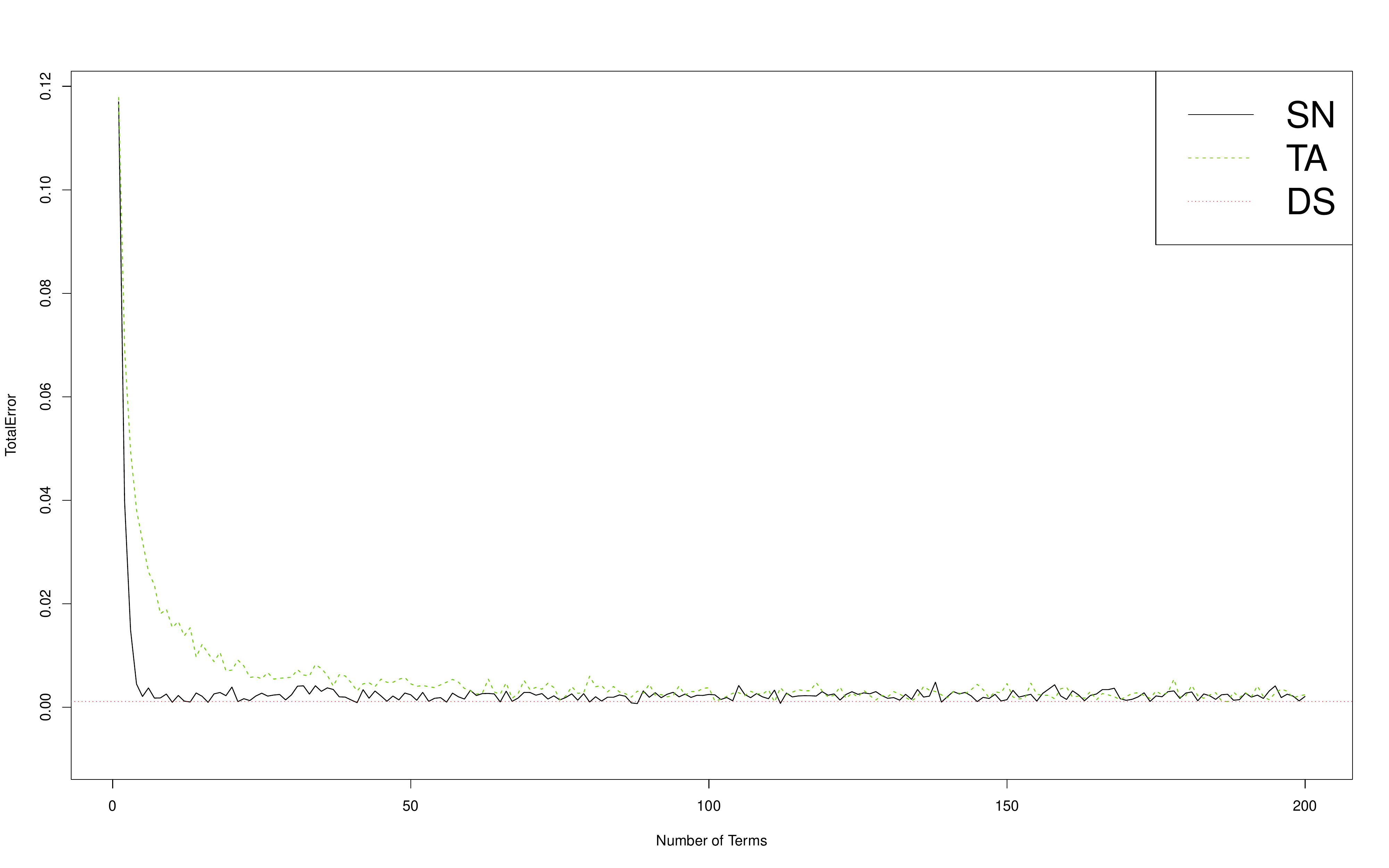} 
\caption{Plots of errors when $\sigma$ has $50$ evenly spaced directions. The $x$-axis represents $k$, the number of terms in the sum, and the $y$-axis represents the errors. Since DS is exact in this case, it is presented as a baseline.}\label{fig: error in sims discrete}
\end{figure}

\begin{figure}
\center
\begin{tabular}{cc}
{\footnotesize SN} & {\footnotesize TA} \vspace{-.35cm}\\
\includegraphics[width=0.45\textwidth]{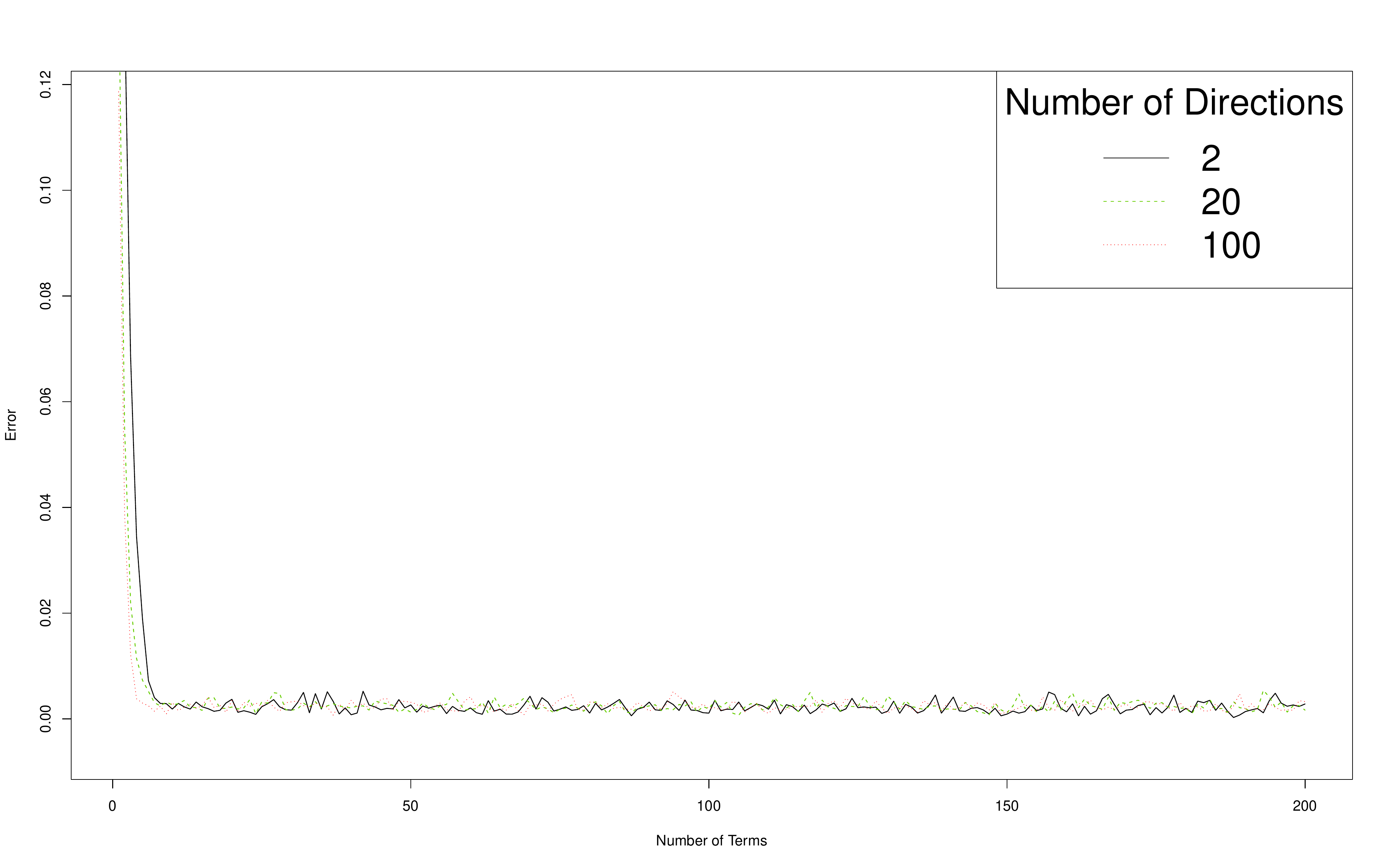} & \includegraphics[width=0.45\textwidth]{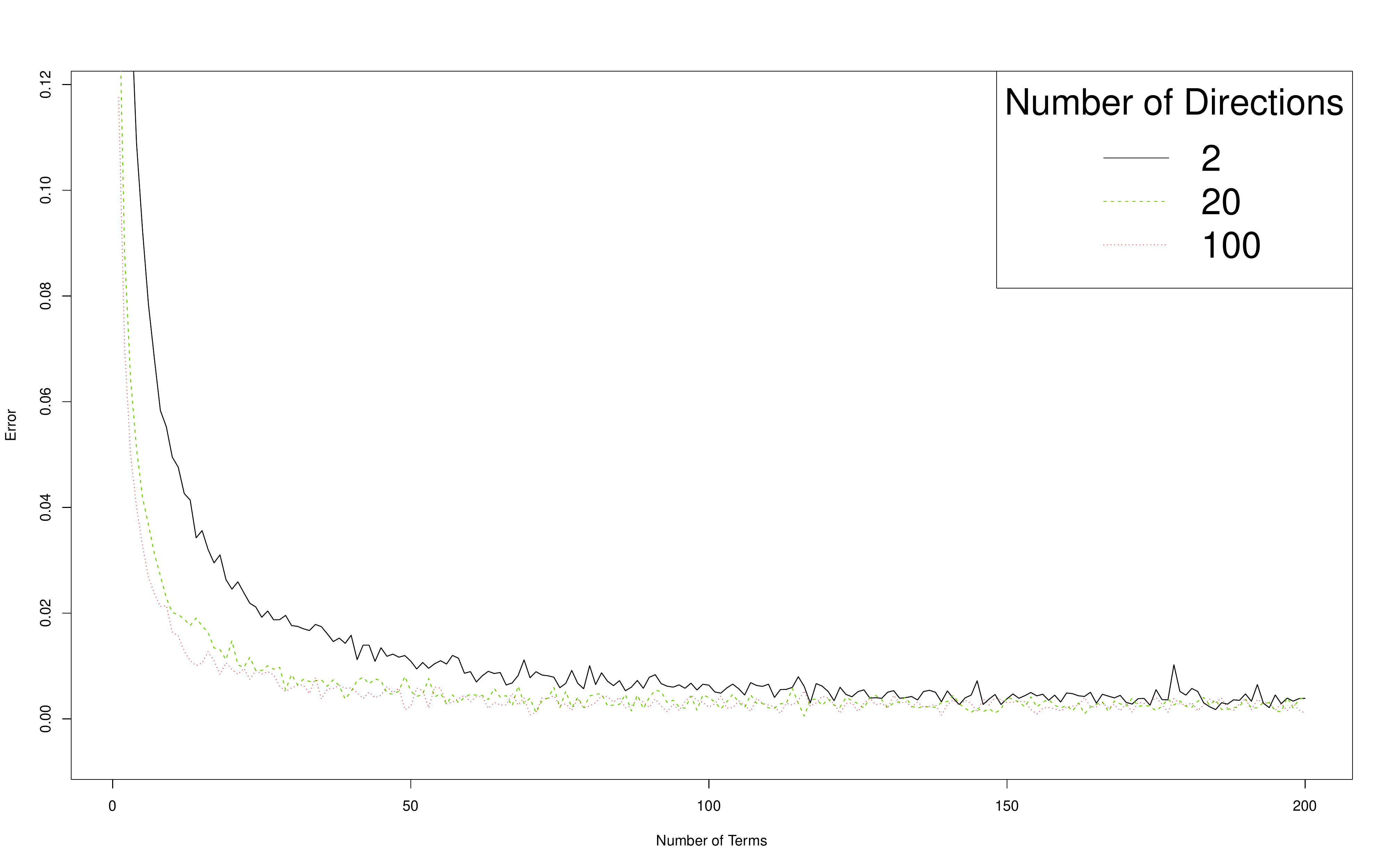}
\end{tabular}
\caption{Plots of errors when $\sigma$ has $r=2,20,100$ evenly spaced directions. The plot of the left is for the SN method and the one on the right is for the TA method. The $x$-axis represents $k$, the number of terms in the sum, and the $y$-axis represents the errors. DS is not presented as it is exact in this case.}\label{fig: error in sims discrete diff dir}
\end{figure}

Overall, SN converges quickly in all of the situations that we considered. It is also easy to implement. TA is also easy to implement, but it needs more terms to converge. DS is an exact method  when $\sigma$ has finite support.  However, it is slightly harder to implement as it requires one to simulate GD random variables, which is a bit more involved. 

\section{Proofs}\label{sec: proofs}

\begin{proof}[Proof of Lemma \ref{lemma: moments}]
We get \eqref{eq: M alpha 2} from \eqref{eq: M alpha} by change of variables. Next, \eqref{eq: M alpha 2}  implies that $M_\alpha(\mathbb R^d) =\int_{\mathbb R^d}\int_0^\infty r^{\alpha-1}\rd r\nu(\rd x)$ and so  $M_\alpha(\mathbb R^d)=\infty$ for any $\nu\ne0$. For the next part, note that, for $p>0$,  l'H\^opital's rule gives
$$
\int_a^\infty r^{\alpha-1} e^{-rp}\rd r\sim  p^{-1}a^{\alpha-1} e^{-ap} \ \ \mbox{as} \ \ a\to\infty.
$$
Thus, there exists a $K\ge e$ such that, if $|y|>K$, then
$$
\frac{1}{2p}(\log|y|)^{\alpha-1} e^{-p\log|y|}\le \int_{\log|y|}^\infty r^{\alpha-1} e^{-rp}\rd r\le  \frac{2}{p}(\log|y|)^{\alpha-1} e^{-p\log|y|}.
$$
Hence,
\begin{eqnarray*}
\int_{|x|\le1}|x|^pM_\alpha(\rd x) &=& \int_{\mathbb R^d}|y|^p\int_0^{|y|^{-1}\wedge1} (-\log r)^{\alpha-1}r^{p-1}\rd r \nu(\rd y)\\
&=&\int_{\mathbb R^d}|y|^p\int_{-\log(|y|^{-1}\wedge1)}^\infty r^{\alpha-1}e^{-pr}\rd r \nu(\rd y)\\
&\le& \int_{|y|\le K}|y|^p\nu(\rd y)\int_0^\infty r^{\alpha-1}e^{-pr}\rd r  + \frac{2}{p}\int_{|y|>K} (\log|y|)^{\alpha}\nu(\rd y)
\end{eqnarray*}
and similarly
\begin{eqnarray*}
\int_{|x|\le1}|x|^pM_\alpha(\rd x) &\ge& \int_{|y|\le K}|y|^p\nu(\rd y)\int_{\log K}^\infty u^{\alpha-1}e^{-pr}\rd r  + \frac{1}{2p}\int_{|y|>K} (\log|y|)^{\alpha}\nu(\rd y).
\end{eqnarray*}
For the last part, we have
\begin{eqnarray*}
\int_{|x|>1}|x|^pM_\alpha(\rd x) &=& \int_{|y|>1}|y|^p\int_{|y|^{-1}}^1 (-\log r)^{\alpha-1}r^{p-1}\rd r \nu(\rd y)\\
&=&\int_{|y|>1}|y|^p\int_0^{\log|y|} r^{\alpha-1}e^{-pr}\rd r \nu(\rd y)\\
&\le&\frac{\Gamma(\alpha)}{p^\alpha} \int_{|y|>1}|y|^p\nu(\rd y)
\end{eqnarray*}
and
\begin{eqnarray*}
\int_{|x|>1}|x|^pM_\alpha(\rd x) &=&\int_{|y|>1}|y|^p\int_0^{\log|y|} r^{\alpha-1}e^{-pr}\rd r \nu(\rd y)\\
&\ge&\int_{|y|>e}|y|^p \nu(\rd y)\int_0^{1} r^{\alpha-1}e^{-pr}\rd r,
\end{eqnarray*}
which completes the proof.
\end{proof}

\begin{proof}[Proof of Theorem \ref{thrm: integ rep}.]
Let $M_\alpha$ be as in \eqref{eq: M alpha} and let $C_1$ be the cgf of $Y_1$. By Corollary 2.3 in \cite{Sato:2006}, the stochastic integral is absolutely definable so long as we have the finiteness of 
\begin{eqnarray*}
 &&\int_0^\infty \left|C_1\left(z e^{-(\alpha s/\theta)^{1/\alpha}} \right)\right| \rd s\\
&&\qquad\le |z|| \gamma' | \int_0^\infty e^{-(\alpha s/\theta)^{1/\alpha}} \rd s + \theta^{-1} \int_0^\infty \int_{\mathbb{R}^d}1 \wedge(|z||x|e^{-(\alpha s/\theta)^{1/\alpha}}) \nu(\rd x) \rd s\\
&&\qquad\le |z|| \gamma' | \int_0^\infty e^{-(\alpha s/\theta)^{1/\alpha}} \rd s + \int_{\mathbb{R}^d}2 \wedge(|z||x|) M_\alpha (\rd x) <\infty,
\end{eqnarray*}
where we use \eqref{eq: M alpha 2}, \eqref{eq: finite for MVP alpha}, and the fact that $|e^{ia}-1|\le2\wedge|a|$ for $a\in\mathbb R$, see Section 26 in \cite{Billingsley:1995}. Similarly, by Proposition 2.2 in \cite{Sato:2006}, the cgf of the stochastic integral is given by
\begin{eqnarray*}
&&\int_0^\infty C_1\left(z e^{-(\alpha s/\theta)^{1/\alpha}} \right) \rd s\\
&&\qquad =\int_0^\infty\left(i  \langle \gamma', z \rangle e^{-(\alpha s/\theta)^{1/\alpha}} +
	\int_{\mathbb{R}^d}\left(e^{i\left\langle z e^{-(\alpha s/\theta)^{1/\alpha}}, x\right\rangle } - 1\right) \nu'(\rd x) \right) \rd s\\
&&\qquad= i \langle \gamma', z \rangle \int_0^\infty e^{-(\alpha s/\theta)^{1/\alpha}} \rd s +
	\int_{\mathbb{R}^d}\left(e^{i\left\langle z, x\right\rangle } - 1\right) M_\alpha (\rd x).
\end{eqnarray*}
From here, the result follows from the fact that $ \int_0^\infty e^{-(\alpha s/\theta)^{1/\alpha}} \rd s=\theta\Gamma(\alpha)$.
\end{proof}

\begin{proof}[Proof of Theorem \ref{thrm: series rep gen Vervaat}]
The result follows by a general shot-noise representation of L\'evy processes given in \cite{Rosinski:2001}, see also Theorem 6.2 in \cite{Cont:Tankov:2004}. Define
$$
H(r,y) = ye^{-(r/\theta)^{1/\alpha}}, \ \ r>0,\ y\in\mathbb R^d
$$
and note that $|H(r,y)|$ is nonincreasing in $r$ for each $y$. Let
$$
\lambda(r,B) = P(H(r,Y_1)\in B) = P(Y_1\in e^{(r/\theta)^{1/\alpha}}B), \ \ r>0, \ B\in\mathfrak B(\mathbb R^d),
$$
and
$$
A(s) = \int_0^s \int_{|y|\le1} y \lambda(r,\rd y)\rd r, \ \ s\ge0.
$$
From \eqref{eq: M alpha 2} it follows that for $B\in\mathfrak B(\mathbb R^d)$ 
\begin{eqnarray*}
\int_0^\infty \lambda(r,B) \rd r &=& \int_0^\infty \int_{\mathbb R^d} 1_{B}(ye^{-(r/\theta)^{1/\alpha}}) \nu_1(\rd y)\rd r 
=M_\alpha(B).
\end{eqnarray*}
and by dominated convergence that
\begin{eqnarray*}
\lim_{s\to\infty}A(s) &=& \int_0^\infty \int_{|y|\le e^{(r/\theta)^{1/\alpha}}} ye^{-(r/\theta)^{1/\alpha}} \nu_1(\rd y)\rd r = \int_{|x|\le1}xM_\alpha(\rd x).
\end{eqnarray*}
We can use dominated convergence since by \eqref{eq: finite for MVP alpha}
\begin{eqnarray*}
\int_0^s \int_{|y|\le e^{(r/\theta)^{1/\alpha}}} |y|e^{-(r/\theta)^{1/\alpha}} \nu_1(\rd y)\rd r 
&\le& \int_0^\infty \int_{|y|\le e^{(r/\theta)^{1/\alpha}}} |y|e^{-(r/\theta)^{1/\alpha}} \nu_1(\rd y)\rd r \\
&=& \int_{|x|\le1}|x|M_\alpha(\rd x)<\infty.
\end{eqnarray*}
From here Theorem 6.2 in \cite{Cont:Tankov:2004} gives the result for $T=1$.  The results for general $T$ follows from the fact that the L\'evy process $\{X_t:0\le t\le T\}$ where $X_1\sim\ID_0(M,\gamma)$ has the same distribution as $\{X'_t:0\le t\le 1\}$ where $X'_1\sim\ID_0(TM,T\gamma)$. The result for $\alpha=1$ follows from the well-known and easily checked fact that $e^{-E_i}\sim U(0,1)$.
\end{proof}

\begin{proof}[Proof of Theorem \ref{thrm: conv of powers}]
To prove this result, it suffices to verify the conditions for convergence of sums of infinitesimal triangular arrays. Such conditions can be found in, e.g., \cite{Sato:1999}, \cite{Meerschaert:Scheffler:2001}, or \cite{Kallenberg:2002}. A version of these is as follows:\\
1. For any $C\in\mathscr B(\mathbb S^{d-1})$ with $\nu\left(\left\{y\in\mathbb R^d\setminus\{0\}: \frac{y}{|y|}\in\partial C\right\}\right)=0$ and $s>0$, then 
$$
\lim_{n \to \infty} N_n \rP\left(|T_1|X_1^n>s, \frac{T_1}{|T_1|}\in C\right) = M_\alpha\left(\left\{x\in\mathbb R^d:|x|>s, \frac{x}{|x|}\in C\right\}\right).
$$
2.
$$
\lim_{\epsilon\downarrow0}\limsup_{n\to\infty} N_n \rE\left[X_1^n |T_1| 1(X_1^n |T_1| <\epsilon)\right] =0.
$$

We begin by noting that, by L'H\^opital's rule, for any $t>s>0$ we have
$$
n\left(1-(s/t)^{1/n} \right) \sim \log(t/s).
$$
Let $\ell_0(t) = \ell(1/t)$ and note that $\ell_0$ is slowly varying at $\infty$, i.e.\ for every $t>0$
$$
\lim_{x\to\infty} \frac{\ell_0(xt)}{\ell_0(x)}=1.
$$
Proposition 2.6 in \cite{Resnick:2007} implies that for $t>s>0$ we have
$$
\ell\left(1-(s/t)^{1/n} \right)=\ell_0\left(\frac{n}{n(1-(s/t)^{1/n})} \right) \sim \ell_0\left(n \right) = \ell(1/n)
$$
and hence
\begin{eqnarray*}
\lim_{n\to\infty} N_nP(X_1 >(s/t)^{1/n}) = \lim_{n\to\infty} cn^\alpha (1-(s/t)^{1/n})^\alpha \frac{ \ell(1-(s/t)^{1/n})}{\ell(1/n)}
= c (\log(t/s))^{\alpha}.
\end{eqnarray*}
By Theorem 20.3 in \cite{Billingsley:1995}
\begin{eqnarray*}
\lim_{n\to\infty}  N_n \rP\left(|T_1|X_1^n>s, \frac{T_1}{|T_1|}\in C\right) &=& \lim_{n\to\infty} \int_{|y|>s, \frac{y}{|y|}\in C}N_nP(X_1 >(s/|y|)^{1/n}) \nu_0(\rd y)\nonumber\\
    &=& \int_{|y|>s, \frac{y}{|y|}\in C}  \lim_{n\to\infty} N_n P(X_1 >(s/t)^{1/n}) \nu_0(\rd y)\nonumber\\
        &=& c \int_{|y|>s, \frac{y}{|y|}\in C} \left(\log(|y|/s)\right)^\alpha  \nu_0(\rd y)\nonumber\\
        &=&M_\alpha\left(\left\{x\in\mathbb R^d:|x|>s, \frac{x}{|x|}\in C\right\}\right),
\end{eqnarray*}
where in the second line we interchange limit and integration using dominated convergence. To see that we can do this, assume that $n$ is large enough that $N_nn^{-\alpha}\ell(1/n)\le 2c$, let $t>s>0$, $a = \log(t/s)>0$, and fix $\delta\in(0,\alpha)$.  By the Potter bounds, see e.g.\ Theorem 1.5.6 in \cite{Bingham:Goldie:Teugels:1987}, there exists a constant $A>0$ with
\begin{eqnarray}\label{eq: log bound}
N_nP(X_1 >(s/t)^{1/n}) &\le& 2c \left(\frac{1-(s/t)^{1/n}}{1/n}\right)^\alpha \frac{\ell(1-(s/t)^{1/n})}{\ell(1/n)} \nonumber\\
&=& 2c \left(\frac{1-e^{-a/n}}{1/n}\right)^\alpha \frac{\ell(1-e^{-a/n})}{\ell(1/n)}\nonumber\\
&\le& A \left(\left(\frac{1-e^{-a/n}}{1/n}\right)^{\alpha+\delta}\vee  \left(\frac{1-e^{-a/n}}{1/n}\right)^{\alpha-\delta}\right)\nonumber\\
&\le& A \left(( \log(t/s))^{\alpha+\delta}\vee ( \log(t/s))^{\alpha-\delta}\right)\le C_s t^\gamma,
\end{eqnarray}
where $C_s$ is some constant depending on $s$ and we use the fact that $1-e^{-x}\le x$.  Next consider
\begin{eqnarray*}
&&\lim_{\epsilon\downarrow0}\limsup_{n\to\infty} N_n \rE\left[X_1^n |T_1| 1_{[X_1^n |T_1| <\epsilon]}\right] \\
&&\qquad=\lim_{\epsilon\downarrow0}\limsup_{n\to\infty} N_n \int_0^\epsilon P(X_1^n |T_1| 1_{[X_1^n |T_1| <\epsilon]}>s)\rd s\\
&&\qquad\le\lim_{\epsilon\downarrow0}\limsup_{n\to\infty} N_n \int_0^\epsilon P(X_1^n |T_1| >s)\rd s\\
&&\qquad =\lim_{\epsilon\downarrow0} \limsup_{n\to\infty} \int_0^\epsilon \int_{s\le |y|} N_nP(X_1^n >s/|y|)\nu_0(\rd y)\rd s\\
&&\qquad \le A\lim_{\epsilon\downarrow0}\int_{|y|>0} \int_0^{\epsilon} \left(( \log(|y|/s))^{\alpha+\delta}\vee ( \log(|y|/s))^{\alpha-\delta}\right)\rd s \nu_0(\rd y)\\
&&\qquad =A\lim_{\epsilon\downarrow0}\int_{|y|>0} \int_{\log(|y|/\epsilon)}^\infty \left(s^{\alpha+\delta}\vee s^{\alpha-\delta}\right)e^{-s}\rd s |y|\nu_0(\rd y)\\
&&\qquad \le A\lim_{\epsilon\downarrow0}\epsilon^{1-\gamma} \int_{|y|>0} \int_{\log(|y|/\epsilon)}^\infty \left(s^{\alpha+\delta}\vee s^{\alpha-\delta}\right)e^{-s\gamma}\rd s |y|^\gamma \nu_0(\rd y)\\
&&\qquad \le A\lim_{\epsilon\downarrow0}\epsilon^{1-\gamma} \int_{\mathbb R^d} |y|^\gamma \nu_0(\rd y) \int_{0}^\infty \left(s^{\alpha+\delta}\vee s^{\alpha-\delta}\right)e^{-s\gamma}\rd s =0,
\end{eqnarray*}
where the fifth line follows by \eqref{eq: log bound} and the sixth by change of variables.
\end{proof}

\end{document}